\documentclass[11pt]{amsart}
\usepackage{amssymb}
\usepackage{dsfont}

\newtheorem{theorem}{Theorem}[section]
\newtheorem{lemma}[theorem]{Lemma}

\newtheorem{corollary}[theorem]{Corollary}

\theoremstyle{definition}

\newcommand{\I}{\mathcal I}
\newcommand{\A}{\mathcal A}

\newcommand{\Pe}{\mathcal P}
\newcommand{\N}{\mathbb N}

\newcommand{\M}{\mathfrak M}

\newcommand{\on}{\operatorname}

\author{Marek Balcerzak}
\address{Institute of Mathematics, \L \'od\'z University of Technology,
W\'olcza\'nska 215, 93-005 \L \'od\'z, Poland}
\email {marek.balcerzak@p.lodz.pl}
\author{Micha{\l} Pop\l awski}
\address{Institute of Mathematics, \L \'od\'z University of Technology,
W\'olcza\'nska 215, 93-005 \L \'od\'z, Poland}
\email {michal.poplawski.m@gmail.com}
\author{Artur Wachowicz}
\address{Institute of Mathematics, \L \'od\'z University of Technology,
W\'olcza\'nska 215, 93-005 \L \'od\'z, Poland}
\email {artur.wachowicz@p.lodz.pl}
\title[Ideal convergent subsequences]{Ideal convergent subsequences and rearrangements for divergent sequences of functions}
\date{}

\date{}

\begin{document}
\begin{abstract}
Let $\I$ be an ideal on $\N$ which is either analytic or coanalytic.
Assume that $(f_n)$ is a sequence of functions with the Baire property from a Polish space $X$ into a complete metric space
$Z$, which is divergent on a comeager set. 
We investigate the Baire category of $\I$-convergent subsequences and rearrangements of $(f_n)$.  Our result generalizes a theorem of Kallman.
A similar theorem for subsequences is obtained if $(X,\mu)$ is a $\sigma$-finite complete measure space
and a sequence $(f_n)$ of measurable functions from $X$ to $Z$ is $\I$-divergent $\mu$-almost everywhere.
Then the set of subsequences of $(f_n)$, $\I$-divergent $\mu$-almost everywhere, is of full product measure on $\{ 0,1\}^\N$. Here we  assume additionally that $\mathcal I$ has property (G).
\end{abstract}

\maketitle

\section{Introduction}

Denote $\N:=\{ 1,2,\dots\}$.
Let 
$$
S:=\left\{ s\in\N^\N\colon \forall n\in\N\;\; s(n)<s(n+1)\right\}\quad\mbox{and}\quad P:=\left\{ p\in\N^\N\colon p\mbox{ is a bijection}\right\} .
$$
 These are  $G_\delta$ subsets of $\N^\N$,
hence they are Polish spaces. Given a sequence $(x_n)$ in a space $X$, and a property (Prop) of sequences, one can study the Baire category of subsequences and rearrangements of $(x_n)$ with the property (Prop). This can be done by
checking the Baire category of the sets
\begin{equation} \label{EQ1} 
\left\{ s\in S\colon (x_{s(n)})\mbox{ has (Prop)}\right\}\quad \textrm{and}\quad\left\{ p\in P\colon (x_{p(n)})\mbox{ has (Prop)}\right\} 
\end{equation}
in the spaces $S$ and $P$. 
In \cite{BkGW}, these sets were investigated under the assumption that $(x_n)$ is a divergent sequence of reals and (Prop) means $\I$-convergence where $\I$ is an ideal of subsets of $\N$. The paper \cite{BkGW} was motivated by the article by Miller and Orhan \cite{MO} dealing with statistical convergence, a special case of $\I$-convergence. In Section 2, we continue similar studies in an extended form.
Namely, we move to problems concerning sequences of functions with the Baire property, from a Polish space to a complete metric space.
We will assume that a given sequence $(f_n)$ of functions is divergent on a comeager set, and then we will examine the Baire category
of the sets of $\I$-convergent subsequences and rearrangements of $(f_n)$. This approach follow the ideas of Kallman \cite{K} where the Baire category of convergent
(in the usual sense) subsequences, under the same assumption on $(f_n)$, was studied.

The first of the sets in (\ref{EQ1}) can be also investigated from the measure point of view. Namely,
one can introduce a probability measure on $S$ as follows. Let $T\subset \{0,1\}^\N$ consist of
0-1 sequences with infinitely many terms equal to 1. There is a natural homeomorphism $h$ from $S$ onto $T$ which assigns to any $s\in S$, a sequence $t\in T$ such that $t(s(i)):=1$ for all $i\in\N$, and $t(j):=0$ for all $j\notin\{s(i)\colon i\in\N\}$. Consider the uniform probability measure on $\{0,1\}$ and the respective product measure $\nu$ 
on the completion $\mathcal D$ of the respective product $\sigma$-algebra on $\{0,1\}^\N$. Since $T$ is a cocountable subset of $\{0,1\}^\N$, we have $\nu(T)=1$.
We can transfer this measure to $S$ via the bijection $h^{-1}$. More precisely, put $\lambda(E):=\nu(h[E])$
for all $E$ in the $\sigma$-algebra $\mathcal A:=\{ h^{-1}[D]\colon D\in\mathcal D\}$ of subsets of $S$.
In Section 3, for special ideals $\I$ on $\N$, we will state that
$$\lambda(\left\{ s\in S\colon (f_{s(n)})\mbox{ is }\I\mbox{-divergent }\mu\mbox{-a.e.}\right\})=1,$$
provided that a sequence of $\mu$-measurable functions $f_n$, $n\in\N$, from $X$ to a complete metric space $Z$, is $\I$-divergent $\mu$-almost everywhere. The measure $\mu$ is assumed complete and $\sigma$-finite. This result is in a sense similar to that obtained in Section 2, however we will observe some asymmetry between these theorems. 

The reasonings in Sections 2 and 3 use the Kuratowski-Ulam theorem and the Fubini theorem, respectively. 
This idea is borrowed from \cite{K}. Theorem \ref{TW}, proved below, is an important fact, useful in Section 2. Some results of the recent paper \cite{BkGW} are exploited in Section 3.

Let us recall some basic notions connected with ideals on $\N$.
We say that an ideal $\mathcal{I}\subset \Pe(\N)$ is {\em admissible} if $\N\notin\mathcal{I}$ and the ideal $\on{Fin}$ of all finite subsets of $\N$ is contained in $\mathcal{I}$. From now on, we will consider only admissible ideals; we will simply call them {\em ideals on $\N$}. 

If $\mathcal{I}$ is an ideal on $
\N$, we say (cf. \cite{KSW}) that a sequence $(x_n)$ in $(X,d)$ is {\em $\mathcal{I}$-convergent} to $x\in X$
(and write $\I$-$\lim_n x_n=x$) if for every $\varepsilon >0$ we have $\{n\in\N\colon d(x_n,x)\geq\varepsilon\}\in\mathcal{I}$. An $\I$-limit of $(x_n)$ is unique, if it exists. Note that, if $\lim_n x_n=x$ then $\I$-$\lim_n x_n=x$, and 
for $\mathcal{I}:=\on{Fin}$, we obtain the usual convergence of $(x_n)$ to $x$.
In the case when $\mathcal{I}$ equals $\mathcal{I}_d$, {\em the density ideal}
which consists of sets $A\subset\N$ with asymptotic density zero (that is, $d(A):=\lim_{n} |A\cap [1,n]|/n=0$), we speak about {\em statistical convergence} (see \cite{Fas}, \cite{Fr}). 
For $\I\neq\on{Fin}$, it can happen that some subsequences and rearrangements of a divergent sequence are $\I$-convergent
(cf. \cite{KSW}), so it is natural to ask how often (for instance, in the sense of the Baire category) such a phenomenon holds. This motivates the studies of \cite{MO} and \cite{BkGW}. We say that a sequence $(x_n)$ is {\em $\I$-divergent} if it is not $\I$-convergent. For some applications of $\I$-convergence in real analysis, see \cite{BDK}, \cite{BGW}, 
\cite{FMRS},  \cite{LR}, \cite{Mr}.

Ideals on $
\N$ can be treated (via the characteristic functions) as subsets of the Polish space $\{0,1\}^\N$,
so they may have the Baire property, be Borel, analytic, coanalytic, and so on.
Several examples of ideals on $\N$ are presented in \cite{KSW} and \cite{Far}.
The following result, due to Jalali-Naini and Talagrand (see \cite[Theorem 1, Section 8]{To}) gives a characterization of ideals on $\N$ with the Baire property. Recall that sets with {\em the Baire property}, in a given metric space, form the smallest $\sigma$-algebra
containing open sets and meager sets in this space. 

\begin{lemma}\label{Tal}
Let $\I$ be an ideal on $\N$. The following conditions are equivalent:
\begin{itemize}
\item $\I$ has the Baire property;
\item there is an infinite sequence $n_1<n_2<\dots$ in $\N$ such that no member of $\I$ contains infinitely many intervals $[n_i,n_{i+1})$ in $\N$.
\end{itemize}
\end{lemma}

The following theorem is a slight generalization of \cite[Theorem 2.1]{BkGW} which was dealing with sequences of real numbers. We use a similar method of proof. However, we provide details for the reader's convenience.

\begin{theorem} \label{TW}
Let $\I$ be an ideal on $\N$ with the Baire property and let $(x_n)$ be a divergent sequence in a metric space $(X,d)$. Then the sets
$$E(\I,(x_n)):=\{ s\in S\colon (x_{s(n)})\mbox{ \em is }\I\mbox{\em -convergent}\}\quad\mbox{and}\quad
R(\I,(x_n)):=\{ p\in P\colon (x_{p(n)})\mbox{ \em is }\I\mbox{\em-convergent}\}$$ 
are meager in $S$ and $P$, respectively.
\end{theorem}
\begin{proof}
Let us show the part concerned with $E(\I,(x_n))$. 
First, consider an easy case when $(x_n)$ does not contain a convergent subsequence. Then $E(\I,(x_n))=\emptyset$.
Indeed, suppose that there exists $s\in E(\I,(x_n))$. Hence $\I$-$\lim_n x_{s(n)}=x$ for some $x\in X$.
But then $\{n\in\N\colon d(x_{s(n)},x)<1/k\}\in\I^\ast$ (where $\I^\ast$ stands for the dual filter of $\I$), for every $k\in\N$, which easily produces a subsequence $(x_{s(n_k)})$ convergent to $x$ in the usual manner. This yields a contradiction.

The remaining case means that there exists a subsequence $(x_{u(n)})$ of $(x_n)$ convergent to some $x\in X$ and
another subsequence $(x_{v(n)})$ which is not convergent to $x$. Clearly, $u,v\in S$. We may assume that there exists $r>0$
such that $d(x_{u(n)},x)\leq r$ and $d(x_{v(n)},x)\geq 2r$ for all $n\in\N$. 

Since $\I$ has the Baire property, by Lemma \ref{Tal}
we can find an infinite sequence $n_1<n_2<\dots$ in $\N$ such that no member of $\I$ contains infinitely many intervals $[n_i,n_{i+1})$. For any $m\in\N$, define $A_m$ as the following set
$$\{ s\in S\colon \exists k\in\N\;(n_k>m \land (\forall j\in [n_k,n_{k+1})\;d(x_{s(j)},x)\leq r) \land 
(\forall j\in [n_{k+1},n_{k+2})\;d(x_{s(j)},x)\geq 2r)) \} .$$
 Note that $\bigcap_{m\in\N}A_m\subset S\setminus E(\I,(x_n))$. Indeed, if $s\in\bigcap_{m\in\N}A_m$, then each of the sets
 $$\{j\in\N\colon \;d(x_{s(j)},x)\leq r\}\quad\mbox{and}\quad\{j\in\N\colon d(x_{s(j)},x)\geq 2r\}$$
 contains infinitely many intervals of the form $[n_k,n_{k+1})$, hence it does not belong to $\I$.
 Thus $(x_{s(n)})$ is not $\I$-convergent. (Indeed, if we suppose that $\I$-$\lim_n x_n =x_0$, the both cases $d(x_0,x)\leq r$ and $d(x_0,x)>r$ lead to a contradiction.) Consequently, it suffices to show that every set $A_m$ is comeager in $S$.
 
 Fix $m\in\N$. We will prove that every open set $U$ from a standard countable base $\mathcal B$ of topology in $S$ (inherited
 from $\N^{\N}$) contains a set $V_U\in\mathcal B$ included in $A_m$. This will demonstrate that $A_m$ contains the dense
 open set $\bigcup_{U\in\mathcal B}V_U$, hence it is comeager.
 
 So, consider a basic open set 
 $$U:=S\cap\left\{s\in\N^\N\colon s\mbox{ extends }(s_1,\dots, s_d)\right\} $$
 where $s_1<s_2<\dots <s_d$ is a fixed sequence.
We may assume that $d\geq m$. Let $k$ be the smallest index with $n_k>d$.
We extend $(s_1,\dots, s_d)$ in three steps.
Firstly, let $s(i):=s_i$ for $i=1,\dots ,d$ and $s(i):=s_d+i-d$ for $i=d+1,d+2,\dots, n_k-1$.
Secondly, pick the smallest index $p_k$ such that $u(p_k)>s(n_k-1)$ and
put $s(i):=u(p_k+i-n_k)$ for $i=n_k,n_k+1,\dots,n_{k+1}-1$.
In the third step, pick the smallest index $q_k$ such that $v(q_k)>s(n_{k+1}-1)$ and
put $s(i):=v(q_k+i-n_{k+1})$ for $i=n_{k+1}, n_{k+1}+1,\dots,n_{k+2}-1$.
Let $\overline{s}:=(s(1),s(2),\dots,s(n_{k+2}-1))$. Then
$$V:=S\cap\left\{ s\in\N^\N\colon s\mbox{ extends }\overline{s}\right\}$$
is an open set contained in $U$. 
This ends the proof for $E(\I,(x_n))$.

The argument for $R(\I,(x_n))$ is analogous. The respective reasoning uses an open set $U$ similar to that considered
above but $S$ is replaced by $P$ and a sequence $(s_1,\dots, s_d)$ is one-to-one. Also, an extension $\overline{s}$
of $(s_1,\dots, s_d)$ should be chosen one-to-one.
\end{proof}

\section{Results on the Baire category}
Recall that a  mapping from a metric space $X$ to a metric space $Z$ is said to have {\em the Baire property} if its preimage of any open set in $Z$ has the Baire property in $X$.
The following lemma belongs to mathematical folklore.

\begin{lemma} \label{LL1}
Let $f$ be a function from a Polish space $X$ to a Polish space $Y$, and $f$ has the Baire property.
Then, for every analytic (coanalytic) set $A\subset Y$, the preimage $f^{-1}[A]$ has the Baire property.
\end{lemma}
\begin{proof} It suffices to show the assertion if $A$ is analytic.
We can express $A$ as the result of the Suslin operation \cite[Theorem 4.1.13]{Sr}
as follows $A=\bigcup_{z\in\N^{\N}}\bigcap_{n\in\N}A_{z|n}$
where $\{ A_s\colon s\in\bigcup_{n\in\N}\N^n\}$ is a family of closed subsets of $X$ and $z|n:=(z(1),\dots, z(n))$.
Then $f^{-1}[A]=\bigcup_{z\in\N^{\N}}\bigcap_{n\in\N}f^{-1}[A_{z|n}]$ where the sets $f^{-1}[A_{z|n}]$ have
the Baire property in $X$. Since the $\sigma$-algebra of sets with the Baire property is stable under the Suslin
operation \cite[Example 3.5.21, Theorem 3.5.22]{Sr}, so $f^{-1}[A]$ has the Baire property in $X$.
\end{proof}

\begin{lemma} \label{LL2}
Let $(f_n)$ be a sequence of functions, with the Baire property, from a Polish space $(X,d)$ to a complete metric space $(Z,\rho)$.
Assume that $\I$ is an ideal on $\N$ which is either analytic or coanalytic. Then the set
$$B:=\{(x,s) \in X \times S \colon (f_{s(n)}(x)) \mbox{ is } \I\mbox{-convergent}\}$$
has the Baire property.
\end{lemma}
\begin{proof}
Since $(Z,\rho)$ is complete, by the theorem of Dems \cite{D}, the $\I$-convergence of $(f_{s(n)}(x))$ is equivalent to the $\I$-Cauchy condition
$$\forall{\varepsilon>0} \ \exists{N \in \N} \ \{n \in \N \colon \rho(f_{s(n)}(x),f_{s(N)}(x))>\varepsilon\} \in \I.$$
Hence, the set $B$ can be expressed in the form
$$B=\bigcap_{k \in \N} \bigcup_{N \in \N} \{(x,s) \in X \times S \colon \{n \in \N \colon \rho(f_{s(n)}(x),f_{s(N)}(x))>{1}/{k}\} \in \I\}.$$
For fixed $k,N \in \N$ define $g_{kN} \colon X \times S \to \{0,1\}^{\N}$ by the formula
$$g_{kN}(x,s):=(\chi_{\{n \in \N \colon \rho(f_{s(n)}(x),f_{s(N)}(x))>{1}/{k}\}}(j))_{j \in \N}.$$

{\bf Claim 1.}
For every $x \in X$, the function $g_{kN}(x, \cdot)$ is continuous on $S$.

Indeed, since the spaces $\N^{\N}$ and $\{0,1\}^{\N}$ are equipped with the product topologies, it suffices to show that for any sequence $(s_m)$ in $S$ and a point $s \in S$ such that 
$$\lim_{m\to\infty}s_m(n) = s(n)\quad\mbox{for all }n\in\N ,$$ 
we have
$$\lim_{m\to\infty}g_{kN}(x,s_m)(j)= g_{kN}(x,s)(j)\quad\mbox{for all }j \in \N .$$
 Fix $j\in\N$.
Since $\N$ is discrete, $s_m(j)=s(j)$ and $s_m(N)=s(N)$ for all $m \geq m_0$ where $m_0 \in \N$ is sufficiently large. For these numbers $m \in \N$ we have $g_{kN}(x,s_m)(j)=g_{kN}(x,s)(j)$.   

{\bf Claim 2.} 
For every $s \in S$, the function $g_{kN}(\cdot,s)$ has the Baire property.
 
Indeed, fix $j \in \N$ and let $U_i:=\{w \in \{0,1\}^{\N} \colon w(j)=i\}$ where $i \in \{0,1\}$. It suffices to prove that the preimage of $U_i$ with respect to $g_{kN}(\cdot,s)$ has the Baire property. But this preimage equals either
$$\{x \in X \colon \rho(f_{s(j)}(x),f_{s(N)}(x)) \leq {1}/{k}\}\quad \mbox{if } i=0,$$
$$\mbox{or}\quad\{x \in X \colon \rho(f_{s(j)}(x),f_{s(N)}(x)) > {1}/{k}\}\quad \mbox{if } i=1.$$
Note that the mapping $x\mapsto \rho(f_{s(j)}(x),f_{s(N)}(x))$ has the Baire property, as the respective composition with
a continuous function.
Hence, in the both cases, the considered preimage has the Baire property.

By \cite[Theorem 3.1.30]{Sr} and Claims 1,2, we infer that the function $g_{kN}$ has the Baire property. Note that $B=g_{kN}^{-1}[\I]$. Since $\I$ is analytic or coanalytic, $B$ has the Baire property by Lemma \ref{LL1}.
\end{proof}

\begin{theorem} \label{TW2}
Let $(f_n)$ be a sequence of functions, with the Baire property, from a Polish space $(X,d)$ to a complete metric space $(Z,\rho)$. Assume that $\I$ is an ideal on $\N$ which is either analytic or coanalytic.
Then the following conditions are equivalent:
\begin{itemize}
\item[(i)] the set 
$\{x \in X \colon (f_n(x)) \mbox{ is divergent}\}$ 
is comeager in $X$;
\item[(ii)] the set $\{x \in X \colon\{s \in S \colon (f_{s(n)}(x)) \mbox{ is } \I\mbox{-convergent}\} \mbox{ is meager}\}$ is comeager in $X$;
\item[(iii)] the set $\{(x,s) \in X \times S \colon (f_{s(n)}(x)) \mbox{ is } \I\mbox{-convergent}\}$ 
is meager in $X\times S$;
\item[(iv)] the set $\{s \in S \colon \{x \in X \colon (f_{s(n)}(x)) \mbox{ is } \I\mbox{-convergent}\} \mbox{ is meager}\}$ is comeager in $S$.
\end{itemize}
\end{theorem}
\begin{proof}
Denote the  sets in the statements (i) and (ii) by $A$ and $H$, respectively.
To prove (i)$\Rightarrow$(ii) assume that $A$ is comeager. If $x\in A$ then, by Theorem \ref{TW}, the set $E(\I,(f_n(x)))$ is meager in $S$. Hence the set $H$ in the statement (ii) contains $A$, so it is comeager in $X$.

To show (ii)$\Rightarrow$(i) assume that $H$ is comeager. Let $x\in X\setminus A$. Then for every $s\in S$, the sequence $(f_{s(n)})$ is convergent,
hence $x\in X\setminus H$. So $A$ is comeager.

Let $B$ stand for the set in the statement (iii). Denote by $B_x$ and $B^s$ the respective sections of $B$, if $x\in X$ and $s\in S$. We have
 $$B_x=\{s \in S \colon (f_{s(n)}(x)) \mbox{ is } \I\mbox{-convergent}\}\quad\mbox{and}\quad
 B^s=\{x \in X \colon (f_{s(n)}(x)) \mbox{ is } \I\mbox{-convergent}\}.$$
 Thanks to Lemma \ref{LL2}, $B$ has the Baire property.
 Consequently, the equivalences (ii)$\Leftrightarrow$(iii)$\Leftrightarrow$(iv) follow from
 the Kuratowski-Ulam theorem \cite{Ox} and its converse.
\end{proof}

Observe that the analogue of Theorem \ref{TW2} for rearrangements (with $S$ replaced by $P$) is true, and
the proof is similar.

The implications (i)$\Rightarrow$(iv) in the both theorems yield the following corollary.

\begin{corollary} \label{CC1}
Let $(f_n)$ be a sequence of functions, with the Baire property, from a Polish space $(X,d)$ to a complete metric space $(Z,\rho)$. Assume that $\I$ is an ideal on $\N$ which is either analytic or coanalytic. If the sequence $(f_n)$ is divergent on a comeager set in $X$,
then the sets
$$\{s \in S \colon \{x \in X \colon (f_{s(n)}(x)) \mbox{ is } \I\mbox{-convergent}\} \mbox{ is meager}\};$$
$$\{p \in P \colon \{x \in X \colon (f_{p(n)}(x)) \mbox{ is } \I\mbox{-convergent}\} \mbox{ is meager}\},$$
are comeager in $S$ and $P$, respectively.
\end{corollary}

The above corollary for subsequences generalizes a theorem by Kallman \cite[Thm. 3.1]{K}, where the case
$\I=\on{Fin}$ with $Z$ being a separable Banach space, was considered. A technique using the Kuratowski-Ulam theorem appears in the proofs of the both results. In fact, the result of Kallman and its proof were an inspiration for our studies in this paper. Our reasoning is similar, however we had to overcome more difficulties.

\section{A result for subsequences in the measure case}
We will use the probability measure space $(S, \mathcal A, \lambda)$ introduced in Section 1. Let us start from the following measure counterparts of Lemmas \ref{LL1} and \ref{LL2}.

\begin{lemma} \label{LK1}
Let $(W,\mathfrak N, \tau)$ be a $\sigma$-finite complete measure space and let $f$ be an $\mathfrak N$-measurable function from $W$ to a Polish space $Y$.
Then, for every analytic (coanalytic) set $A\subset Y$, the preimage $f^{-1}[A]$ belongs to $\mathfrak N$.
\end{lemma}

The proof is analogous to that of Lemma \ref{LL1} -- we use the fact that the $\sigma$-algebra $\M$ is stable under the Suslin operation \cite[Example 3.5.20, Theorem 3.5.22]{Sr}.

\begin{lemma} \label{LK2}
Let $(X, \M,\mu)$ be a $\sigma$-finite complete measure space and
let $(f_n)$ be a sequence of $\M$-measurable functions from $X$ to a complete metric space $(Z,\rho)$.
Assume that $\I$ is an ideal on $\N$ which is either analytic or coanalytic. Then the set
$$B:=\{(x,s) \in X \times S \colon (f_{s(n)}(x)) \mbox{ is } \I\mbox{-convergent}\}$$
belongs to the completion $\M\otimes\A$ of the product $\sigma$-algebra of $\M$ and $\A$.
\end{lemma}

The proof is analogous to that of Lemma \ref{LL2}. In the final part, we use \cite[Theorem 3.1.30]{Sr} and we infer that the respective function $g_{kN}$ is $\M\otimes\A$-measurable. Then, by Lemma \ref{LK1}, the set $B=g_{kN}^{-1}[\I]$ is in $\M\otimes\A$.

If $\I$ is an ideal on $\N$, a function $f\colon\N\to\N$ is called {\em $\I$-invariant} if $A\in\I\Leftrightarrow f[A]\in\I$ for every $A\subset\N$ (see \cite{BGS}). We say that $\I$ has {\em property} (G) if 
$\lambda(\{s\in S\colon s \mbox{ is }\I\mbox{-invariant}\})=1$ (see \cite{BkGW}).
Note that the density ideal $\I_d$ has property (G); some other examples are given in \cite{BkGW}.

The following fact was proved in \cite[Theorem 3.4]{BkGW} for a sequence of reals, however it can be easily generalized to a sequence in any metric space, with the same proof.

\begin{lemma} \cite[Theorem 3.4]{BkGW} \label{LK3}
Assume that $\I$ is an analytic or coanalytic ideal on $\N$, having property (G). For a sequence $(z_n)$ in a metric space
$(Z,\rho)$, the following conditions are equivalent:
\begin{itemize}
\item[(I)] $(z_n)$ is $\I$-convergent;
\item[(II)] $\lambda (\{s\in S\colon (z_{s(n)}) \mbox{ is }\I\mbox{-convergent}\})=1$.
\end{itemize}
\end{lemma}

The following theorem is a main result of this section.

\begin{theorem} \label{TW3}
Let $(X, \M,\mu)$ be a $\sigma$-finite complete measure space and
let $(f_n)$ be a sequence of $\M$-measurable functions from $X$ to a complete metric space $(Z,\rho)$.
Assume that $\I$ is an ideal on $\N$ which is either analytic or coanalytic and has property (G).
Then the following conditions are equivalent:
\begin{itemize}
\item[(i)] the set 
$\{x \in X \colon (f_n(x)) \mbox{ is }\I\mbox{-divergent}\}$ 
is of full measure $\mu$ in $X$;
\item[(ii)] the set $\{x \in X \colon\lambda(\{s \in S \colon (f_{s(n)}(x)) \mbox{ is } \I\mbox{-convergent}\})=0 \}$ is of full measure $\mu$ in $X$;
\item[(iii)] $(\mu\times\lambda)(\{(x,s) \in X \times S \colon (f_{s(n)}(x)) \mbox{ is } \I\mbox{-convergent}\})=0$;
\item[(iv)] $\lambda(\{s \in S \colon \mu(\{x \in X \colon (f_{s(n)}(x)) \mbox{ is } \I\mbox{-convergent}\})=0 \})=1$. 
\end{itemize}
\end{theorem}
\begin{proof} 
Denote the sets in the statements (i) and (ii) by $A$ and $H$, respectively.
To prove (i)$\Rightarrow$(ii) assume that $A$ is of full measure $\mu$. Let $x\in A$. By Lemma \ref{LK3}, the set 
$$E(x):=\left\{ s\in S\colon (f_{s(n)}(x))\mbox{ is }\I\mbox{-convergent}\right\}$$ 
satisfies $\lambda(E(x))<1$. The image $h[E(x)]$ under the
homeomorphism $h\colon S\to T$ (see Section 1) is a tail set in $\{0,1\}^\N$, so by the 0-1 law \cite[Theorem 21.3]{Ox} we have $\lambda(E(x))=0$.
Hence $x\in H$. Consequently, $H$ is of full measure $\mu$.

To show (ii)$\Rightarrow$(i) assume that $H$ is of full measure $\mu$. Let $x\in X\setminus A$. Then 
$(f_n(x))$ is $\I$-convergent, so by Lemma \ref{LK3} we have $\lambda(E(x))=1$ where $E(x)$ is as above.
Hence $x\in X\setminus H$. So $A$ is of full measure $\mu$.

Let $B$ stand for the set in the statement (iii). Thanks to Lemma \ref{LK2}, we have $B\in\M\otimes\A$. Thus, the equivalences
(ii)$\Leftrightarrow$(iii)$\Leftrightarrow$(iv) follow from
 the Fubini theorem and its converse.
\end{proof}

We have the following analogue of Corollary \ref{CC1}. 

\begin{corollary} \label{CK1}
Let $(X, \M,\mu)$ be a $\sigma$-finite complete measure space and
let $(f_n)$ be a sequence of $\M$-measurable functions from $X$ to a complete metric space $(Z,\rho)$.
Assume that $\I$ is an ideal on $\N$ which is either analytic or coanalytic and has property (G).
If the sequence $(f_n)$ is $\I$-divergent $\mu$-almost everywhere, then  
$\lambda(\{s \in S \colon \mu(\{x \in X \colon (f_{s(n)}(x)) \mbox{ is } \I\mbox{-convergent}\})=0 \})=1$.
\end{corollary}

Observe an asymmetry between the category and the measure cases.
Namely, the divergence of $(f_n)$ on a large set is stated in condition (i) of Theorem \ref{TW2},
while, in the respective condition of Theorem \ref{TW3}, we have the $\I$-divergence of $(f_n)$ on a large set.
The remaining conditions (ii)--(iv) are completely analogous.


\begin{thebibliography}{abc}
\bibitem{BDK} M. Balcerzak, K. Dems, A. Komisarski, {\it Statistical convergence and ideal convergence for sequences of functions}, J. Math. Anal. Appl. {\bf 328} (2007), 715--729.
\bibitem{BGS} M. Balcerzak, Sz. G{\l}\c{a}b, J. Swaczyna, {\it Ideal invariant injections}, submitted.
\bibitem{BkGW} M. Balcerzak, Sz. G{\l}\c{a}b, A. Wachowicz, {\it Qualitative properties of ideal convergent subsequences and rearrangements}, submitted.
\bibitem{BGW} A. Bartoszewicz, Sz. G{\l}\c{a}b, A. Wachowicz, {\it Remarks on ideal boundedness, convergence and variation of sequences}, J. Math. Anal. Appl. {\bf 375} (2011), 431--435.
\bibitem{D} K. Dems, {\it On $\I$-Cauchy sequences}, Real Anal. Exchange {\bf 30} (2004/2005), 123--128.
\bibitem{Far} I. Farah, {\it Analytic quotients. Theory of lifting for quotients over analytic ideals on integers}, Mem. Amer. Math. Soc. {\bf 148} (2000).
\bibitem{Fas} H. Fast, {\it Sur la convergence statistique}, Colloq. Math. {\bf 2} (1951), 241--244.
\bibitem{FMRS} R. Filip\'ow, N. Mro\.zek, I. Rec{\l}aw, P. Szuca, \emph{\it Ideal convergence of bounded sequences}, J. Symb. Logic {\bf 72} (2007), 501--512. 
\bibitem{Fr} J.A. Fridy, {\it On statistical convergence}, Analysis {\bf 5} (1985), 301--313.
\bibitem{K} R.R. Kallman, {\it Subsequences and category}, Internat. J. Math. {\bf 22} (1999), 709--712.
\bibitem {KSW} P. Kostyrko,  T. \v{S}al\'at, W. Wilczy\'nski, {\it $\mathcal I$-convergence}, Real Anal. Exchange {\bf 26} (2000-2001), 669--685.
\bibitem{LR} M. Laczkovich, I. Rec{\l}aw, {\it Ideal limits of sequences of continuous functions}, Fund. Math. {\bf 203} (2009), 39--46.
\bibitem{MO} H.I. Miller, C. Orhan, {\it On almost convergent and statistically convergent subsequences}, Acta Math. Hungar. {\bf 93} (2001), 135--151.
\bibitem {Mr} N. Mro\.zek, {\it Ideal version of Egorov's theorem for analytic P-ideals}, J. Math. Anal. Appl. {\bf 349} (2009), 452--458.
\bibitem {Ox} J.C. Oxtoby, {\it Measure and category}, Springer, New York 1980.
\bibitem{Sr} S.M. Srivastava, {\it A course of Borel sets}, Springer, New York 1998.
\bibitem {To} S. Todorcevic, {\it Topics in topology}, Lecture Notes in Math. {\bf 1652}, Springer, New York 1997.
\end{thebibliography}
\end{document}